\documentclass{article}
\usepackage{amsmath}
\usepackage{amssymb}
\usepackage{amsfonts,color}
\usepackage{graphicx}

\newtheorem{theorem}{Theorem}[section]

\newtheorem{corollary}[theorem]{Corollary}

\newtheorem{lemma}[theorem]{Lemma}

\newtheorem{proposition}[theorem]{Proposition}
\newtheorem{remark}[theorem]{Remark}

\newenvironment{proof}[1][Proof]{\noindent\textbf{#1.} }{\ \rule{0.5em}{0.5em}}
\input{tcilatex}
\begin{document}

\title{H\"{o}lder Conditions of Local Times and Exact Moduli of
non-differentiability for Spherical Gaussian fields}
\author{Xiaohong Lan \\
School of Mathematical Sciences \\
University of Science and Technology of China \\
E-mail: xhlan@ustc.edu.cn\\
\and Yimin Xiao \\
Department of Statistics and Probability\\
Michigan State University\\
E-mail: xiao@stt.msu.edu}
\maketitle

\begin{abstract}
This paper investigate the local times and modulus of nondifferentiability
of the spherical Gaussian random fields. We extend the methods for studying
the local times of Gaussian to the spherical setting. The new main
ingredient is the property of strong local nondeterminism established
recently in Lan et al (2018).
\end{abstract}



\section{Introduction and statement of main results}


In this paper, we shall be concerned with the local time of an isotropic
Gaussian random field $\mathbf{T}=\left\{ \mathbf{T}\left( x\right) ,\,x\in%
\mathbb{S}^{2}\right\} $ with values in $\mathbb{R}^{d}$ defined on some
probability space $\left( \Omega ,\Im, \mathbb{P}\right)$ by 
\begin{equation}  \label{Def:T}
\mathbf{T}\left( x\right) =\left( T_{1}\left( x\right) ,\ldots
,T_{d}\left(x\right) \right) ,\quad x\in \mathbb{S}^{2},
\end{equation}
where $T_{1},\ldots ,T_{d}$ are independent copies of $T_{0}=\left\{T_{0}%
\left( x\right) ,\text{ }x\in \mathbb{S}^{2}\right\} $, which is a zero-mean
Gaussian random field that satisfies 
\begin{equation}  \label{Eq:T0}
\mathbb{E}\left[ T_{0}\left( x\right) T_{0}\left( y\right) \right] =\mathbb{E%
}\left[ T_{0}\left( gx\right) T_{0}\left( gy\right) \right]
\end{equation}%
for all $g\in SO\left( 3\right) $ and all $x,y\in \mathbb{S}^{2}$. We call $%
T_{0}$ the Gaussian field associated with $\mathbf{T}.$

It follows from \eqref{Eq:T0} that the Gaussian field $T_{0}$ is both
2-weakly and strongly isotropic (cf. \cite{MPbook} p.121). 
It is well-known (cf. \cite{MPbook}, p.123) that $T_{0}$ has the following
spectral representation:%
\begin{equation}  \label{rapT}
T_{0}(x) =\sum_{\ell \geq 0}\sum_{m=-\ell }^{\ell }a_{\ell m}Y_{\ell m}( x) ,
\end{equation}%
where 
\begin{equation*}
a_{\ell m}=\int_{\mathbb{S}^{2}}T_{0}\left( x\right) \overline{Y}_{\ell
m}\left( x\right) d\nu \left( x\right) ,
\end{equation*}%
and $\nu $ is the canonical area measure on the sphere. In the spherical
coordinates $(\vartheta ,\varphi )\in \lbrack 0,\pi ]\times \lbrack 0,2\pi )$%
, $\nu (dx)=\sin \vartheta d\vartheta d\varphi $. Note that the equality in %
\eqref{rapT} holds in both the $L^{2}\left( \mathbb{S}^{2}\times \Omega,d\nu
\left( x\right) \otimes \mathbb{P}\right) $ sense and the $L^{2}\left( 
\mathbb{P}\right) $ sense for every fixed $x\in \mathbb{S}^{2}$. The set of
homogenous polynomials $\left\{ Y_{\ell m}:\ell \geq 0,m=-\ell ,...,\ell
\right\} $ are the spherical harmonic functions on $\mathbb{S}^{2},$
representing an orthonormal basis for the space $L^{2}\left( \mathbb{S}%
^{2},\nu \right) $. In this setting, the random coefficients $\{a_{\ell m},\
\ell \geq 0,m=-\ell ,...,\ell \}$ are Gaussian complex random variables such
that 
\begin{equation*}
\begin{split}
\mathbb{E}\left[ a_{\ell m}\right] & =0\text{;} \\
\mathbb{E}\left[ a_{\ell m}\overline{a}_{\ell _{1}m_{1}}\right] &
=\delta_{\ell }^{\ell _{1}}\delta _{m}^{m_{1}}C_{\ell }\text{,}
\end{split}%
\end{equation*}
where the sequence $\left\{ C_{\ell }\right\} $ of nonnegative numbers is
called the angular power spectrum of $T_{0}$, which fully characterizes the
dependence structures of $T_{0}$. A celebrated theorem of Schoenberg \cite%
{schoenberg1942} provides the following expansion for the covariance
function: 
\begin{equation*}
\mathbb{E}\left[ T_{0}\left( x\right) T_{0}\left( y\right) \right]
=\sum_{\ell =0}^{+\infty }C_{\ell }\frac{2\ell +1}{4\pi }P_{\ell
}\left(\left\langle x,y\right\rangle \right) ,
\end{equation*}
where for $\ell =0,1,\,2,...,$ $P_{\ell }:[-1,1]\rightarrow \mathbb{R}$
denotes the Legendre polynomial satisfying the normalization condition $%
P_{\ell }(1)=1$. Thus for every $x\in \mathbb{S}^{2},$ 
\begin{equation*}
\mathbb{E}\left[ T_{0}\left( x\right) ^{2}\right] =\sum_{\ell
=0}^{+\infty}C_{\ell }\frac{2\ell +1}{4\pi }:=K.
\end{equation*}

For simplicity, we assume in this paper that $K=1.$ Otherwise, we consider
the random field $\{K^{-1/2}T_{0}\left( x\right) ,\text{ }x\in \mathbb{S}%
^{2}\}$, which does not cause any essential loss of generality.

As shown by Lang and Schwab \cite{lang1}, Lan et al. \cite{LanMarXiao} that
the degree of smoothness of the sample paths of $T_{0}$ can be precisely
characterized by the asymptotic behavior of the angular power spectrum $%
\{C_{\ell }\}$ at high multipoles $\ell >>0.$ In this paper, we further
illustrate this point by investigating regularity properties of the local
times of the vector-valued random field $\mathbf{T}$ in \eqref{Def:T}. 
For this purpose, we recall the following condition from Lan et al. \cite%
{LanMarXiao}:

\textbf{Condition (A)} The random field $T_{0}$ is a zero-mean, Gaussian and
isotropic random field indexed by $\mathbb{S}^{2}$, with unit variance and
angular power spectrum satisfying: 
\begin{equation*}
C_{\ell }=C_{\ell }(G,\alpha )=G\left( \ell \right) \ell ^{-\alpha }>0,
\end{equation*}
where $\alpha >2$ and there exists a constant $K_{0}\geq 1$ such that for
all $\ell =1,2,...$ 
\begin{equation*}
K_{0}^{-1}\leq G\left( \ell \right) \leq K_{0}.
\end{equation*}

The regularity of the trajectories of the Gaussian field $T_{0}$ is governed
by $\alpha $. It has been shown in \cite{lang2,LanMarXiao,lang1} that the
sample function of $T_{0}$ is almost surely $k$-times continuous
differentiable if and only if $\alpha >2+2k$. Hence, the field $T_{0}$ is
twice differentiable almost surely (as needed for the Kac-Rice argument in 
\cite{AdlerTaylor}, Theorem 12.1.1) if and only if $\alpha >6$.

In this paper we focus on the non-smooth regime $2<\alpha <4$. In this case,
the exact uniform and local moduli of continuity have recently been proved
by Lan et al. \cite{LanMarXiao}. We remark that the regime of $2<\alpha <4$
is indeed the most relevant for many areas of applications; in particular, a
major driving force for the analysis of spherical random fields has been
provided over the last decade by cosmological applications, for instance in
connection to the analysis of Cosmic Microwave Background radiation data
(CMB). Data analysis on CMB maps is currently an active research area, and
the geometry of CMB maps has been deeply investigated (see \cite%
{planck2013a, planck2013c, planck2015c}). In this framework there are strong
theoretical and empirical evidence to support the belief that the values of $%
\alpha$ belong to the non-smooth region $2<\alpha <4$.

Our objective of the present paper is to establish the joint continuity, and
the uniform and local moduli of continuity for the local times of the
vector-valued random field $\mathbf{T}$. Based on these results, we show
that the sample functions of $\mathbf{T}$ and $T_{0}$ are a.s. nowhere
differentiable, and we determine the exact modulus of non-differentiability
of $\mathbf{T}$.

In order to state our main results, we need some notations. In spherical
coordinates $(\vartheta ,\varphi )\in \lbrack 0,\pi ]\times \lbrack 0,2\pi )$%
, every point $x\in {\mathbb{S}}^{2}$ can be written as $x=(\sin
\vartheta\cos \varphi ,\sin \vartheta \sin \varphi ,\cos \vartheta )$. In
this paper, we always identify the Cartesian and spherical coordinates of $%
x\in {\mathbb{S}}^{2}$. For any $x\in \mathbb{S}^{2},$ $0<r<\pi ,$ let $%
D(x,r)=\left\{y\in \mathbb{S}^{2}:d_{\mathbb{S}^{2}}(x,y)<r\right\} $ be an
open disk on $\mathbb{S}^{2},$ where $d_{\mathbb{S}^{2}}(x,y)=\arccos
\left\langle x,y\right\rangle $ denotes the standard geodesic distance on
the sphere, and $\left\langle \cdot ,\cdot \right\rangle $ is the inner
product in $\mathbb{R}^{3}$. Given $(\vartheta ,\varphi )\in \lbrack 0,\pi
]\times \lbrack 0,2\pi)$, denote by $V(\vartheta ,\varphi )=\{(\theta ,\phi
):0\leq \theta \leq\vartheta ,0\leq \phi \leq \varphi \}$ the angular
section originated from the North Pole. It is easy to see that,
respectively, the spherical area of $D(x,r)$ and $V(\vartheta ,\varphi )$ are

\begin{equation}
\nu (D(x,r))=2\pi (1-\cos r)\ \ \ \text{ and }\ \ \ \nu (V(\vartheta
,\varphi ))=\varphi (1-\cos \vartheta ).  \label{Areas}
\end{equation}%
Our first theorem is concerned with the joint continuity of local times of
the $\mathbf{T}$, see Section \ref{Sec:Joint continuity} below for the
definition of local times and more information.

\begin{theorem}
\label{Th1} Let $\mathbf{T}=\left\{ \mathbf{T}(x),x\in \mathbb{S}%
^{2}\right\} $ be a Gaussian random field with values in $\mathbb{R}^{d}$
defined in $\left( \ref{Def:T}\right) $. Assume that the associated
isotropic random field $T_{0}$ satisfies Condition (A) with $2<\alpha <4$
and $\beta =4-\left( \alpha -2\right) d>0$. Then for any open set $%
D\subseteq \mathbb{S}^{2}\ $with $\nu (D)>0$, $\mathbf{T}$ has local times $%
L\left( \mathbf{t},D\right) $ which is in $\mathbb{R}^{d}\times \mathbb{S}%
^{2}$ almost surely. Moreover, there is a modification of $L\left( \mathbf{t}%
,D\right) $ such that it is jointly continuous in the following sense:

\begin{itemize}
\item[(i)] The local time $L\left( \mathbf{t},D(x,r)\right) $ is continuous
in $(\mathbf{t},x,r)$. Moreover, for any two open disks $D_{i}=D%
\left(x_{i},r_{i}\right) \subseteq \mathbb{S}^{2}$ with $x_{i}\in \mathbb{S}%
^{2},r_{i}\in \left( 0,\delta \right) ,\ i=1,2,$ and any $\gamma \in
\left(0,\gamma _{0}\right) $ with $\gamma _{0}=\min \left\{ \frac{\beta }{%
2\left(\alpha -2\right) },1\right\} ,$ it satisfies

\begin{equation*}
\left\vert L\left( \mathbf{t},D_{1}\right) -L\left( \mathbf{s}%
,D_{2}\right)\right\vert \leq K_{1}\left[ \left\Vert \mathbf{t}-\mathbf{s}%
\right\Vert^{\gamma }+d_{\mathbb{S}^{2}}\left( x,y\right) ^{\beta
/4}r^{\gamma \left(\alpha -2\right) -\beta /4}\right] r^{\eta }\quad a.s.
\end{equation*}
where $\left\Vert \mathbf{\cdot }\right\Vert $ is the Euclidean distance on $%
\mathbb{R}^{d}$ and 
\begin{equation}  \label{eta}
\eta =\frac{\beta }{2}-\left( \alpha -2\right) \gamma >0;
\end{equation}

\item[(ii)] The local time $L\left( \mathbf{t},V(\vartheta ,\varphi )\right)$
is continuous in $(\mathbf{t},\vartheta ,\varphi )$. Moreover, for any two
angular sections $V_{i}=V(\vartheta _{i},\varphi _{i})\subseteq \mathbb{S}%
^{2}$ with $(\vartheta _{i},\varphi _{i})\in \lbrack 0,\pi ]\times \lbrack
0,2\pi ),\ i=1,2,$ and any $\gamma \in \left( 0,\gamma _{0}\right),$ it
satisfies 
\begin{eqnarray*}
\lefteqn{\left\vert L\left( \mathbf{t},V_{1}\right) -L\left( \mathbf{s}%
,V_{2}\right) \right\vert} \\
& \leq& K_{2}\left[ \left\vert \varphi _{1}-\varphi _{2}\right\vert \min
\left\{\vartheta _{1}^{2},\vartheta _{2}^{2}\right\} +\min \left\{
\varphi_{1},\varphi _{2}\right\} \left\vert \vartheta
_{1}^{2}-\vartheta_{2}^{2}\right\vert \right]^{\beta /4} \\
&&+K_{2}\left\Vert \mathbf{t}-\mathbf{s}\right\Vert ^{\gamma}\varphi
_{1}^{\eta /2}\vartheta _{1}^{\eta }, \quad a.s.
\end{eqnarray*}
where the constant $\delta \in \left( 0,1\right) $ depending only on $%
\alpha\ K_{0},$ and $K_{1},K_{2}$ are positive depending only on $\alpha ,\
d $ $K_{0}$ and $\gamma .$
\end{itemize}
\end{theorem}

The next result provides optimal upper and lower bounds for the exact moduli
of continuity for the maximum of local time $L^{\ast }(D)=\sup_{\mathbf{t}%
\in \mathbb{R}^{d}}L\left( \mathbf{t},D\right) $ \textit{w.r.t} to the
variable $r$ in the set $D=D(x,r)$.

\begin{theorem}
\label{Th2} Under conditions of Theorem \ref{Th1}, there exist positive
constant $K_{3},K_{4}$ depending only on $\alpha ,\ d$ and $K_{0},$ such
that for any $x\in \mathbb{S}^{2},$ 
\begin{equation}
K_{3}^{-1}\leq \underset{r\rightarrow 0}{\lim \inf }\frac{L^{\ast }\left(
D\left( x,r\right) \right) }{\phi \left( r\right) }\leq \underset{%
r\rightarrow 0}{\lim \sup }\frac{L^{\ast }\left( D\left( x,r\right) \right) 
}{\phi \left( r\right) }\leq K_{3},\ a.s.  \label{ineq:Local Holder of LT}
\end{equation}%
where the functions $\phi $ is defined by 
\begin{equation}
\phi \left( r\right) =\frac{r^{2}}{\left[ \rho _{\alpha }(r/\sqrt{\log
\left\vert \log r\right\vert })\right] ^{d}},  \label{psi1}
\end{equation}%
with $\rho _{\alpha }(r)=r^{\frac{\alpha }{2}-1}$ for $r\geq 0$.
\end{theorem}

\begin{remark}
The result above further comfirms that the sample functions of $T_{0}$ (and $%
\mathbf{T)}$ are nowhere differentiable. Actually, in Section \ref%
{Sec:non-diff},\ we establish the Chung's law of iterated logarithm, which
is essential to the proof of Theorem \ref{Th2} as well.
\end{remark}

The rest of this paper is as follows\emph{: } Section\emph{\ }\ref%
{preliminaries}\emph{\ }reviews some auxiliary tools for the arguments to
follow. Section \ref{Sec:Joint continuity} and \ref{Sec:Holder condition}
present the proofs of Theorems 1.1 and 1.2., respectively. The Chung's law
of iterated logarithm of $T_{0}$ is established in Section \ref{Sec:non-diff}%
. Our method for establishing joint continuity and upper bound of H\"{o}lder
conditions for the local times relies on moment estimates on the local
times, where the property of strong local nondeterminsim of $T_{0}$ proved
in Lan et al. \cite{LanMarXiao}. plays an essential role. We also make use
of a chaining argument that is similar to those in \cite{GeHor, Xiao97}. 

Throughout this paper, we use $C,K$ to denote a constant whose value may
change in each appearance, and $C_{i,j},K_{i,j}$ to denote the $j$th more
specific positive finite constant in Section $i.$ 

\medskip

\textbf{Acknowledgement} The authors thank Professor Domenico Marinucci for
stimulating discussions and helpful comments on this paper. Research of X.
Lan is supported by NSFC grants 11501538 and WK0010450002. Research of Y.
Xiao is partially supported by NSF grants DMS-1612885 and DMS-1607089. 

\section{Preliminaries}

\label{preliminaries}

In this section, we collect a few technical results which will be
instrumental for most of the proofs to follow. We recall first the following
lemma from \cite{LanMarXiao}, which characterizes the variogram and the
property of strong local nondeterminism of $T_{0}$.

\begin{lemma}
\label{C1} Under Condition (A) with $2<\alpha <4$, there exist positive
constants $K_{2,1}\geq 1,$ $0<\delta <1$ depending only on $\alpha $ and $%
K_{0},$ such that for any $x,y\in \mathbb{S}^{2},$ if $d_{\mathbb{S}^{2}}(
x,y) <\delta ,$ we have 
\begin{equation}  \label{ineq:variogram}
K_{2,1}^{-1}\rho _{\alpha }^{2}\left( d_{\mathbb{S}^{2}}( x,y)\right) \leq 
\mathbb{E}\left[ (T_{0}\left( x\right) -T_{0}( y))^{2}\right] \leq
K_{2,1}\rho _{\alpha }^{2}\left( d_{\mathbb{S}^{2}}(x,y) \right) .
\end{equation}
Moreover, there exists a constant $K_{2,2}>0$ depending on $\alpha $ and $%
K_{0}$ only, such that for all integers $n\geq 1$ and all $%
x,x_{1},...,x_{n}\in \mathbb{S}^{2},$ we have 
\begin{equation}  \label{ineq:SLND}
\mathrm{Var}\left( T_{0}\left( x\right) |T_{0}\left(
x_{1}\right),...,T_{0}\left( x_{n}\right) \right) \geq K_{2,2}\min_{1\leq
k\leq n}\rho_{\alpha }^{2}\left( d_{\mathbb{S}^{2}}\left( x,x_{k}\right)
\right).
\end{equation}
\end{lemma}

For any fixed point $x_{0}\in \mathbb{S}^{2},$ consider the spherical random
field $Z_{0}(x)=T_{0}( x) -T_{0}( x_{0}) ,\ x\in \mathbb{S}^{2}.$ By taking
the same argument as in the proof of Theorem 1 in \cite{LanMarXiao}, we
obtain the following consequence of $\left( \ref{ineq:SLND}\right) $ in
Lemma \ref{C1}.

\begin{corollary}
\label{C2'} Under Condition (A) with $2<\alpha <4$, there exists a constant $%
K_{2,2}^{\prime }>0,$ such that for all integers $n\geq 1$ and all $%
x,x_{1},...,x_{n}\in \mathbb{S}^{2},$

\begin{equation*}
\mathrm{Var}\left( Z_{0}\left( x\right) |Z_{0}\left(
x_{1}\right),...,Z_{0}\left( x_{n}\right) \right) \geq K_{2,2}^{\prime
}\min_{0\leq k\leq n}\left\{ \rho _{\alpha }^{2}\left( d_{\mathbb{S}%
^{2}}(x,x_{k}) \right) \right\} .
\end{equation*}
\end{corollary}

The next auxiliary tool that we will use to prove Theorem \ref{Th1} is the
following lemma, where $\left( \cdot \right) ^{T}$ denotes the transpose of
a matrix or a vector.

\begin{lemma}
\label{Charac-Cov} Let $\mathbf{X}=\left( X_{1},...,X_{n}\right) ^{T}\sim
N\left( \boldsymbol{\mu ,\Theta }\right) $ with a positive definite
covariance matrix $\boldsymbol{\Theta }$ and mean vector $\boldsymbol{\mu
\in }\mathbb{R}^{d}\ \left( d\geq 1\right) ,$ then for all vectors $\mathbf{t%
}\in \mathbb{R}^{d},$ we have 
\begin{equation*}
\left( 2\pi \right) ^{-n/2} \int_{\mathbb{R}^{n}} e^{-i\mathbf{t}^{T}%
\boldsymbol{\xi }} \mathbb{E}\left[ e^{i\boldsymbol{\xi }^{T}\mathbf{X}}%
\right] d\boldsymbol{\xi} \leq \left[ \mathrm{Var}\left(
X_{1}\right)\prod\limits_{j=2}^{n}\mathrm{Var}\left(
X_{j}|X_{1},...,X_{j-1}\right) \right] ^{-1/2};
\end{equation*}
Moreover, 
\begin{equation*}
\mathrm{Var}\left( X_{1}\right) \prod\limits_{j=2}^{n}\mathrm{Var}%
\left(X_{j}|X_{1},...,X_{j-1}\right) =\det \left\{ \boldsymbol{\Theta }%
\right\} .
\end{equation*}
\end{lemma}

In order to prove the lower bounds in Theorems \ref{Th2}, we will also need
to exploit the following two lemmas from Talagrand \cite{Talagrand95}. Let $%
\{f(x),x\in M\}$ be a centered Gaussian field indexed by $M$ and let $%
d_{f}(x,y)=\sqrt{\mathbb{E}[(f(x)-f(y))^{2}]}$ be its canonical metric on $M$%
. For a $d_{f}$-compact manifold $M$, denote by $N_{d_{f}}\left(
M,\varepsilon \right) $ the smallest number of balls of radius $\varepsilon $
in metric $d_{f}$ that are needed to cover $M.$

\begin{lemma}
\label{DudleyLB} If $N_{d_{f}}\left(M,\varepsilon \right) \leq \Psi
\left(\varepsilon \right)$ for all $\varepsilon >0$ and the function $\Psi$
satisfies 
\begin{equation*}
\frac{1}{K_{2,3}}\Psi \left( \varepsilon \right) \leq \Psi \left( \frac{%
\varepsilon }{2}\right) \leq K_{2,3}\Psi \left( \varepsilon \right), \quad
\forall \, \varepsilon >0,
\end{equation*}
where $K_{2,3}>0$ is a finite constant. Then 
\begin{equation*}
\mathbb{P}\left\{ \sup_{s,t\in M}\left\vert f(s) -f(t) \right\vert \leq
u\right\} \geq \exp \left( -K_{2,4}\Psi \left(u\right) \right),
\end{equation*}
where $K_{2,4}>0$ is a constant depending only on $K_{2,3}$.
\end{lemma}

\begin{lemma}
\label{DudleyUB} Let $\{f(x), x \in M\} $ be a centered Gaussian field $a.s.$
bounded on a $d_f$-compact set $M$. There exists a universal constant $%
K_{2,5}>0$ such that for any $u>0,$ we have 
\begin{equation*}
\mathbb{P}\left\{ \sup_{x\in M}f(x) \geq K_{2,5}\left( u+\int_{0}^{\overline{%
d}}\sqrt{\log N_{d}\left( M,\varepsilon \right) }d\varepsilon
\right)\right\} \leq \exp \left( -\frac{u^{2}}{\overline{d}^{2}}\right) ,
\end{equation*}
where $\overline{d}=\sup \left\{ d_{f}( x,y): x,y\in M\right\} $ is the
diameter of $M$ in the metric $d_f$.
\end{lemma}

Based on Lemmas \ref{C1} and \ref{DudleyUB}, we obtain the following result:

\begin{lemma}
\label{LemTupper} Under Condition (A) with $2<\alpha <4$, there exists
positive constants $K_{2,6},K_{2,7}$ depending only on $\alpha $ and $K_{0}$
such that for any $D\left( z,r\right) \subset \mathbb{S}^{2}$ and $%
0<r<\delta ,$ we have for any $u>K_{2,6}r^{\left( \alpha -2\right) /2},$ 
\begin{equation}  \label{ineq:Tlower}
\mathbb{P}\left\{ \sup_{x,y\in D\left( z,r\right) }\left\vert
T_{0}\left(x\right) -T_{0}\left( y\right) \right\vert \geq K_{2,7}u\right\}
\geq \exp\left( -\frac{u^{2}}{\left\vert \rho _{\alpha }\left( 2r\right)
\right\vert^{2}}\right) ,
\end{equation}
\end{lemma}

\begin{proof}
Recall Lemma \ref{C1}, we have 
\begin{equation*}
\sqrt{K_{2,1}^{-1}}\rho _{\alpha }( x,y) \leq d_{T_{0}}\left(x,y\right) \leq 
\sqrt{K_{2,1}}\rho _{\alpha }( x,y) .
\end{equation*}
It follows immediately that, for any $D\left( z,r\right) \subset \mathbb{S}%
^{2},$ and any $\epsilon \in \left( 0,r\right) ,$ 
\begin{equation*}
N_{d_{T_{0}}}\left( D\left( z,r\right) ,\epsilon \right) \leq \frac{2\pi
r^{2}}{\pi \left( \epsilon /\sqrt{K_{2,1}}\right) ^{4/\left( \alpha-2\right)
}} \leq 2\left( K_{2,1}\right)^{2/\left( \alpha -2\right) } \frac{r^{2}}{%
\epsilon ^{4/\left( \alpha -2\right) }},
\end{equation*}
and 
\begin{equation*}
\overline{d}=\sup \left\{ d_{T_{0}}\left( x,y\right) :x,y\in
D\left(z,r\right) \right\} \leq \sqrt{K_{2,1}}\rho _{\alpha }\left(
2r\right) ,
\end{equation*}
whence similar to the argument in \cite{LuanXiao12}, we have 
\begin{equation*}
\int_{0}^{\overline{d}}\sqrt{\log N_{d_{T_{0}}}\left( D\left(
z,r\right),\epsilon \right) }d\epsilon \leq C_{2,1}r^{\alpha /2-1},
\end{equation*}
for some positive constant $C_{2,1}$ depending on $K_{2,1}.$ 
Let the constants $K_{2,6}=C_{2,1}$ and $\ K_{2,7}=2K_{2,5}C_{2,1}$, then by
exploiting Lemma \ref{DudleyUB}, we derive $\left( \ref{ineq:Tlower}\right) $%
.
\end{proof}

Finally, we prove the following approximation for small ball probability,
which is analogy to the argument in \cite{Xiao09}, Theorem 5.1.

\begin{lemma}
\label{Lem:Small Ball P} Under Condition (A) with $2<\alpha <4$, there
exists a positive constant $K_{2,8}$ depending only on $\alpha $ and $K_{0},$
such that for any $\epsilon >0$ and $D\left( z,r\right) \subset \mathbb{S}%
^{2}$ with $0<r<\delta $, we have 
\begin{equation*}
\mathbb{P}\left\{ \sup_{x,y\in D( z,r) }\left\vert T_{0}(x) -T_{0}( y)
\right\vert \leq \epsilon \right\} \geq \exp \left( -K_{2,8}\frac{r^{2}}{%
\epsilon ^{4/\left( \alpha -2\right) }}\right) ,
\end{equation*}
\end{lemma}

\begin{proof}
We first prove the lower bound of the small ball probability. The canonical
metric $d_{T_{0}}$ on $\mathbb{S}^{2}$ is defined by 
\begin{equation*}
d_{T_{0}}\left( x,y\right) =\sqrt{\mathbb{E}\left\vert T_{0}\left( x\right)
-T_{0}\left( y\right) \right\vert ^{2}}.
\end{equation*}%
By $\left( \ref{ineq:variogram}\right) $\ in Lemma \ref{C1}, we have for any 
$x,y\in \mathbb{S}^{2}$ with $d_{\mathbb{S}^{2}}\left( x,y\right) <\delta $, 
\begin{equation*}
d_{T_{0}}\left( x,y\right) \leq \sqrt{K_{2,1}}\rho _{\alpha }(x,y).
\end{equation*}%
In the meantime, recall the metric entropy, it follows immediately that 
\begin{equation*}
N_{d_{T_{0}}}\left( D\left( z,r\right) ,\epsilon \right) \leq \frac{2\pi
r^{2}}{\pi \left( \epsilon /\sqrt{K_{2,1}}\right) ^{4/\left( \alpha
-2\right) }}\leq C_{2,2}\frac{r^{2}}{\epsilon ^{4/\left( \alpha -2\right) }}.
\end{equation*}%
where $C_{2,2}$ is a positive constant depending only on $K_{2,1}$ and $%
\alpha .$ Hence, Lemma \ref{Lem:Small Ball P}\ is derived by exploiting
Lemma \ref{DudleyLB} with $K_{4,1}=K_{2,4}C_{2,2}.$
\end{proof}

\section{Local times and their joint continuity}

\label{Sec:Joint continuity}

Let us first recall that, for any Borel set $D\subseteq \mathbb{S}^{2}$ and $%
\omega \in \Omega $, the \emph{occupation measure} of $\mathbf{T}$ on $D$ is
defined by 
\begin{equation*}
\mu _{D}\left( I,\omega \right) :=\nu \left\{ x\in D:\mathbf{T}(x,\omega)
\in I\right\}
\end{equation*}
for all Borel sets $I\subset \mathbb{R}^{d}$, where 
$\nu $ is the canonical area measure on the unit sphere. If $\mu
_{D}(\cdot,\omega )$ is absolutely continuous \emph{w.r.t} the Lebesgue
measure $\lambda _{d}$ on $\mathbb{R}^{d}$, then we say that $\mathbf{T}%
(\cdot,\omega )$ has \emph{local times} on $D$, and define a local time $%
L\left( \mathbf{t},D,\omega \right) $ as the Radon--Nikod\'{y}m derivative
of $\mu_{D}$ with respect to $\lambda _{d}$, i.e., 
\begin{equation}  \label{def:Local time}
L\left( \mathbf{t},D,\omega \right) =\frac{d\mu _{D}(\cdot ,\omega )}{%
d\lambda _{d}}(\mathbf{t}),\qquad \forall \mathbf{t}\in \mathbb{R}^{d}.
\end{equation}
As usual, we will from now on omit $\omega $ from the notation for the local
times.

We refer to Geman and Horowitz \cite{GeHor} for a comprehensive survey on
local times of random fields. In particular, Theorems 6.3 and 6.4 in \cite%
{GeHor} show that $L\left( \mathbf{t},D\right) $ satisfies the following
occupation density identity: For every Borel set $D\subseteq \mathbb{S}^{2}$
and for every measurable function $f:\mathbb{R}\rightarrow \mathbb{R}_{+},$ 
\begin{equation}  \label{eq:LT-occup density}
\int_{D}f\left( \mathbf{T}( x) \right) d\nu ( x) =\int_{\mathbb{R}^{d}}f( 
\mathbf{t}) L( \mathbf{t},D) d\mathbf{t}.
\end{equation}

We split the proof of Theorem \ref{Th1} into two parts.

\begin{proof}[\textbf{Proof of Theorem \protect\ref{Th1}: Existence}]
The Fourier transform of the occupation measure $\mu _{D}$ is 
\begin{equation*}
\widehat{\mu }_{D}\left( \boldsymbol{\xi }\right) =\int_{D}e^{i\boldsymbol{%
\xi }^{T}\mathbf{T}(x)}d\nu (x),\quad \forall \boldsymbol{\xi \in \mathbb{R}}%
^{d}.
\end{equation*}%
Now applying Fubini's theorem and the independence of components of $\mathbf{%
T}$, we derive 
\begin{equation*}
\begin{split}
\mathbb{E}\left[ \int_{\mathbb{R}^{d}}\left\vert \widehat{\mu }\left( \xi
\right) \right\vert ^{2}d\xi \right] & =\int_{D}d\nu \left( x\right)
\int_{D}d\nu \left( y\right) \int_{\mathbb{R}^{d}}\mathbb{E}\left[ e^{i%
\boldsymbol{\xi }^{T}\left( \mathbf{T}\left( x\right) -\mathbf{T}\left(
y\right) \right) }\right] d\boldsymbol{\xi } \\
& =\int \int_{D\times D}\left( \mathbb{E}\left[ \big|{T_{0}}\left( x\right) -%
{T_{0}}\left( y\right) \big|^{2}\right] \right) ^{-d/2}d\nu \left( x\right)
d\nu \left( y\right) \\
& \leq K_{2,1}^{d/2}\int \int_{D\times D}\left[ d_{\mathbb{S}^{2}}(x,y)%
\right] ^{(1-\alpha /2)d}d\nu (x)d\nu (y) \\
& \leq K_{2,1}^{d/2}\int_{0}^{2\pi }\int_{0}^{\pi }\vartheta ^{(1-\alpha
/2)d}\sin \vartheta d\vartheta d\phi <\infty .
\end{split}%
\end{equation*}%
In the above, the first inequality follows from $\left( \ref{ineq:variogram}%
\right) $ in Lemma \ref{C1}\ and the last inequality follows from the
condition that $\alpha \in \left( 2,4\right) $ and $(\alpha -2)d<4$. Hence, 
by the Plancherel theorem (see i.e., \cite{Rudin}, Ch. 9), we see that $%
\mathbf{T}$ a.s. has local times which can be represented in the $L^{2}(%
\mathbb{R}^{d})$-sense as 
\begin{equation}
L\left( \mathbf{t},D\right) =\frac{1}{2\pi }\int_{\mathbb{R}^{d}}e^{-i%
\boldsymbol{\xi }^{T}\mathbf{t}}\int_{D}e^{i\boldsymbol{\xi }^{T}\mathbf{T}%
\left( x\right) }d\nu \left( x\right) d\boldsymbol{\xi }.  \label{repLT}
\end{equation}%
The proof of existence is thus completed.
\end{proof}

An immediate consequence of the above proof of the existence of local times
is the following:

\begin{corollary}
Under the same conditions as in Theorem \ref{Th1}, almost surely the range $%
\mathbf{T}\left( \mathbb{S}^{2}\right) $ has positive Lebesgue measure.
\end{corollary}

In order to prove the joint continuity of local times, we first prove the
moment estimates in Lemmas \ref{L1} and \ref{L2}, by extending the argument
in Xiao \cite{Xiao97} to the spherical random field $\mathbf{T}$. The new
ingredient for the proofs is the property of strong local nondeterminism in
Lemma \ref{C1}. \medskip

\begin{lemma}
\label{L1} Under conditions of Theorem \ref{Th1}, there exists a positive
constant $K_{3,1}$ depending only on $\alpha ,\ d$ and $K_{0},$ such that
for any open set $D\subset \mathbb{S}^{2}$ with $\nu \left( D\right) >0$, $%
\mathbf{t}\in \mathbb{R}^{d}$, and integers $n\geq 1,$ it holds that 
\begin{equation*}
\mathbb{E}\left[ L\left( \mathbf{t},D\right) ^{n}\right] \leq
K_{3,1}^{n}\left( \left( n-1\right) !\right) ^{d\left( \alpha
-2\right)/4}\nu ( D) ^{( n-1) \beta /4+1}.
\end{equation*}
where $\beta $ is defined in Theorem \ref{Th1}.
\end{lemma}

\begin{proof}
It follows from $( \ref{repLT}) $ (see also \cite{GeHor,Xiao09}) that for
all $\mathbf{t}\in \mathbb{R}^{d}$ and integers $n\geq 1,$ 
\begin{equation*}
\mathbb{E}\left[ L\left( \mathbf{t},D\right) ^{n}\right] =\frac{1}{%
\left(2\pi \right) ^{n}} \int_{D^{n}} \int_{\mathbb{R}^{nd}}e^{-i%
\sum_{j=1}^{n}\mathbf{t}^{T}\boldsymbol{\xi }_{j}}\mathbb{E}\left[
e^{i\sum_{j=1}^{n} \boldsymbol{\xi }_{j}^{T}\mathbf{T}\left( x_{j}\right) }%
\right] \prod\limits_{j=1}^{n}\left( d\nu \left( x_{j}\right) `d\boldsymbol{%
\xi }_{j}\right) ,
\end{equation*}
where $\boldsymbol{\xi }_{j}\in \mathbb{R}^{d}$ for $j=1,...,n$. By the
positive definiteness of covariance matrix of $T_{0}\left(
x_{1}\right),...,T_{0}\left( x_{n}\right) ,$ we have 
\begin{eqnarray*}
\lefteqn{\int_{\mathbb{R}^{nd}}\mathbb{E}\left[ e^{i\sum_{j=1}^{n}%
\boldsymbol{\xi }_{j}^{T}\left[ \mathbf{T}( x_{j}) -\mathbf{t}\right] }%
\right] \prod \limits_{j=1}^{n} d\boldsymbol{\xi }_{j}} \\
&=&\prod\limits_{k=1}^{d} \int_{\mathbb{R}^{n}}\mathbb{E}\left[
e^{i\sum_{j=1}^{n}\xi _{j,k}\left[ T_{k}\left(x_{j}\right) -t_{k}\right] }%
\right] \prod\limits_{j=1}^{n}d\xi _{j,k} \\
&\leq &\frac{( 2\pi ) ^{nd/2}}{\left[ \det \left( \mathrm{Cov}\left( T_{0}(
x_{1}) ,...,T_{0}( x_{n}) \right)\right) \right] ^{d/2}} \\
&\leq &( 2\pi ) ^{nd/2}\left[ \mathrm{Var}( T_{0}(x_{1}) )
\prod\limits_{j=2}^{n}\mathrm{Var}( T_{0}(x_{j}) |T_{0}( x_{1}) ,...,T_{0}(
x_{j-1})) \right] ^{-d/2}.
\end{eqnarray*}
in view of Lemma \ref{Charac-Cov}. Recall the assumption of unit variance of 
$T_{0}$ and the inequality $\left( \ref{ineq:SLND}\right) $ in Lemma \ref{C1}%
, we obtain that 
\begin{equation}  \label{ELeq}
\mathbb{E}\left[ L( \mathbf{t},D) ^{n}\right] \leq C_{3,1}^{n-1}\left[
\int_{D^{n}}\prod\limits_{j=2}^{n}\frac{1}{\min_{1\leq i\leq
j-1}\rho_{\alpha }^{d}\left( x_{i},x_{j}\right) }d\boldsymbol{\nu }\right]
\end{equation}
where the constant $C_{3,1}>0$ depends only on $K_{2,2},$ and $d\boldsymbol{%
\nu }=d\nu \left( x_{1}\right) \cdots d\nu \left( x_{n}\right) .$ Now let $%
j\in \left\{ 2,...,n\right\} $ fixed, and define the following sets that are
disjoint except on the boundaries, 
\begin{equation}  \label{GammaD}
\Gamma _{i} =\left\{ x\in D: d_{\mathbb{S}^{2}}( x,x_{i}) =\min \left\{ d_{%
\mathbb{S}^{2}}( x,x_{i^{\prime }}) , i^{\prime}=1,...,j-1\right\} \right\} .
\end{equation}
Observing $D=\cup _{i=1}^{j-1}\Gamma _{i}$, we have 
\begin{eqnarray}
\lefteqn{\int_{D}\frac{1}{\min_{1\leq i\leq j-1}\rho _{\alpha
}^{d}(x_{i},x_{j}) }d\nu ( x_{j}) }  \notag \\
&=&\sum_{i=1}^{j-1}\int_{\Gamma _{j}}\frac{1}{\rho _{\alpha
}^{d}\left(x_{i},x\right) }d\nu \left( x\right) \leq
\sum_{i=1}^{j-1}\int_{0}^{2\pi}\int_{0}^{r_{j}( \phi ) }\frac{\theta }{\rho
_{\alpha}^{d}( \theta ) }d\theta d\phi  \notag \\
&=&\sum_{i=1}^{j-1}\frac{2}{\beta }\int_{0}^{2\pi }\left[ r_{i}\left( \phi
\right) \right] ^{\beta /2}d\phi \leq C_{3,2}\sum_{i=1}^{j-1}\left[%
\int_{0}^{2\pi }\frac{1}{2}\left[ r_{i}\left( \phi \right) \right] ^{2}d\phi %
\right] ^{\beta /4}  \label{ineq:mineq2} \\
&\leq & C_{3,2}\left( j-1\right) \left( \frac{1}{\left( j-1\right) }%
\sum_{i=1}^{j-1}\nu \left( \Gamma _{i}\right) \right) ^{\beta /4} \leq
C_{3,2}\left( j-1\right) \left( \frac{\nu \left( D\right) }{\left(j-1\right) 
}\right) ^{\beta /4}  \label{ineq:mineq3}
\end{eqnarray}
where we have used Jensen's inequality above and the constant $C_{3,2}>0$
depends on $\alpha $ and $d$. Moreover, the inequality in $( \ref%
{ineq:mineq2}) $ holds if and only if $\beta =4-( \alpha-2) d>0.$ It is
readily seen that 
\begin{eqnarray*}
\mathbb{E}\left[ L\left( \mathbf{t},D\right) ^{n}\right] &\leq&
C_{3,1}^{n-1}\int_{D}\left[ \int_{D^{n-1}}\prod\limits_{j=2}^{n}\frac{d\nu (
x_{1}) \cdots d\nu \left( x_{n-1}\right) }{\min_{1\leq i\leq j-1}\rho
_{\alpha }^{d}( x_{i},x_{j}) }\right] d\nu (x_{n}) \\
&\leq & K_{3,1}^{n-1}\left[ \left( n-1\right) !\right] ^{( \alpha-2) d/4}\nu
( D) ^{( n-1) \beta /4+1},
\end{eqnarray*}%
in view of the two results $\left( \ref{ELeq}\right) $ and (\ref{ineq:mineq3}%
), where the constant $K_{3,1}>0$ depends on $\alpha $ and $d$. Hence, the
lemma is proved.
\end{proof}

Recall $\eta =\beta /2-( \alpha -2) \gamma $ defined in $\left(\ref{eta}%
\right) .$ Obviously, $\eta <\beta /2<2$ for any $\gamma \in (0,1) .$ Now we
have the following moment estimation:

\begin{lemma}
\label{L2} Under conditions of Theorem \ref{Th1}, there exists a positive
constant $K_{3,2}$ depending on $\alpha,$ $d\ K_{0\text{ }}$ and $\gamma ,$
such that for any open set $D\subset \mathbb{S}^{2}$ with $\nu
\left(D\right) >0$, and any $\mathbf{s},\mathbf{t}\in \mathbb{R}^{d}$, all
even integers $n\geq 2,\,0<\gamma <1$ satisfying $\eta >0,$ we have 
\begin{equation*}
\mathbb{E}\left\{ \left[ L( \mathbf{t},D) -L( \mathbf{s},D) \right]
^{n}\right\} \leq K_{3,2}^{n}( n!) ^{2-\eta/2}\left\Vert \mathbf{t}-\mathbf{s%
}\right\Vert ^{n\gamma }\nu (D) ^{( n-1) \eta /2+1},.
\end{equation*}
\end{lemma}

\begin{proof}
Recall $( \ref{repLT}) ,$ we have 
\begin{eqnarray*}
\lefteqn{\mathbb{E}\left\{ \left[ L\left( \mathbf{t},D\right) -L\left( 
\mathbf{s},D\right) \right] ^{n}\right\}} \\
&=&\frac{1}{\left( 2\pi \right) ^{n}} \int_{D^{n}}\int_{\mathbb{R}^{nd}}%
\mathbb{E} \left[ e^{i\sum_{j=1}^{n}\boldsymbol{\xi }_{j}^{T}\mathbf{T}
\left( x_{j}\right) }\right] \prod\limits_{j=1}^{n}\left[ e^{-i\mathbf{t}^{T}%
\boldsymbol{\xi }_{j}}-e^{-i\mathbf{s}^{T}\boldsymbol{\xi }_{j}}\right] d\nu
\left( x_{j}\right) d\boldsymbol{\xi }_{j}.
\end{eqnarray*}
By the fact that, for any $\gamma \in ( 0,1) ,$ we have 
\begin{equation*}
\left\vert e^{-i\mathbf{t}^{T}\boldsymbol{\xi }_{j}}-e^{-i\mathbf{s}^{T}%
\boldsymbol{\xi }_{j}}\right\vert = \left\vert e^{-i\left( \mathbf{t}-%
\mathbf{s}\right) ^{T}\boldsymbol{\xi }_{j}}-1\right\vert \leq
2^{1-\gamma}\left\Vert \mathbf{t}-\mathbf{s}\right\Vert ^{\gamma }\left\Vert 
\boldsymbol{\xi }_{j}\right\Vert ^{\gamma },
\end{equation*}
with $j=1,...,n,$ and hence, 
\begin{eqnarray*}
\lefteqn{\mathbb{E}\left\{ \left[ L( \mathbf{t},D) -L( \mathbf{s},D) \right]
^{n}\right\} } \\
&\leq& \frac{2^{n\left( 1-\gamma \right)}\left\Vert \mathbf{t}-\mathbf{s}%
\right\Vert ^{n\gamma }}{\left( 2\pi \right) ^{n}} \int_{D^{n}}\int_{\mathbb{%
R}^{nd}} e^{-\frac{1}{2}\mathrm{Var}\left( \sum_{j=1}^{n}\boldsymbol{\xi }%
_{j}^{T}\mathbf{T} \left( x_{j}\right) \right)}
\prod\limits_{j=1}^{n}\left\Vert \boldsymbol{\xi }_{j}\right\Vert ^{\gamma} d%
\boldsymbol{\xi }_{j}d\nu ( \mathbf{x}) .
\end{eqnarray*}
Since $\left\vert a+b\right\vert ^{\gamma }\leq \left\vert
a\right\vert^{\gamma }+\left\vert b\right\vert ^{\gamma }$ for any real
numbers $a,b$ and $0<\gamma <1,$ we have $\left\Vert \boldsymbol{\xi }%
_{j}\right\Vert^{\gamma }\leq \sum_{k=1}^{d}\left\vert \xi _{j,k}\right\vert
^{\gamma }$, which leads to 
\begin{eqnarray*}
\lefteqn{\int_{\mathbb{R}^{nd}} e^{-\frac{1}{2} \mathrm{Var}( \sum_{j=1}^{n}%
\boldsymbol{\xi }_{j}^{T}\mathbf{T}( x_{j})
)}\prod\limits_{j=1}^{n}\left\Vert \boldsymbol{\xi }_{j}\right\Vert
^{\gamma}d\boldsymbol{\xi }_{j} } \\
&\leq &\sum_{\mathbf{k}\in \left\{ 1,...,d\right\} ^{n}} \int_{\mathbb{R}%
^{nd}} e^{-\frac{1}{2}\mathrm{Var}\left( \sum_{j=1}^{n}\boldsymbol{\xi }%
_{j}^{T}\mathbf{T}\left( x_{j}\right)
\right)}\prod\limits_{j=1}^{n}\left\vert \xi _{j,k_{j}}\right\vert ^{\gamma
}d\boldsymbol{\xi }_{j},
\end{eqnarray*}
with $\mathbf{k=}\left( k_{1},...,k_{n}\right) $ and $k_{j}\in
\left\{1,...,d\right\} $ for $j=1,...,n.$ That is 
\begin{equation}  \label{IntJ}
\left\vert \mathbb{E}\left[ L( \mathbf{t},D) -L( \mathbf{s},D) \right]
^{n}\right\vert \leq \left( 2^{\gamma }\pi \right)^{-n}\left\Vert \mathbf{t}-%
\mathbf{s}\right\Vert ^{n\gamma }\sum_{\mathbf{k}\in \left\{ 1,...,d\right\}
^{n}}\int_{D^{n}}J_{\mathbf{k}}( \mathbf{x}) d\boldsymbol{\nu }
\end{equation}
where $d\boldsymbol{\nu }=d\nu _{1}...d\nu _{n}$ and $J_{\mathbf{k}}$ is the
integral 
\begin{equation*}
J_{\mathbf{k}}\left( \mathbf{x}\right) =\int_{\mathbb{R}^{nd}}e^{-\frac{1}{2}%
\mathrm{Var}\left( \sum_{j=1}^{n}\boldsymbol{\xi }_{j}^{T}\mathbf{T}\left(
x_{j}\right) \right)}\prod\limits_{j=1}^{n}\left\vert \xi
_{j,k_{j}}\right\vert ^{\gamma }d\boldsymbol{\xi }_{j}\ ,
\end{equation*}
for each fixed point $\mathbf{k}$ in the discrete space $\left\{
1,...,d\right\} ^{n}.$ By a generalized H\"{o}lder's inequality and Lemma2.4
in \cite{Xiao97} (see also \cite{Cuzick82}), we see that $J_{\mathbf{k}%
}\left( \mathbf{x}\right) $ is bounded by 
\begin{eqnarray}
\lefteqn{\prod\limits_{j=1}^{n}\left[ \int_{\mathbb{R}^{nd}} e^{-\frac{1}{2}%
\mathrm{Var}( \sum_{j=1}^{n}\sum_{k=1}^{d}\xi_{j,k}T_{k}( x_{j}) )
}\left\vert \xi _{j,k_{j}}\right\vert^{n\gamma }d\boldsymbol{\xi }\right]
^{1/n} }  \notag \\
&=&\frac{\left( 2\pi \right) ^{nd-1} \int_{\mathbb{R}} \left\vert
v\right\vert ^{n\gamma } \exp \left\{ -\frac{v^{2}}{2}\right\} d\nu
\prod\limits_{j=1}^{n}\left( \widetilde{\sigma }_{j}^{2}\right)^{-\gamma /2}%
}{\left[ \det \mathrm{Cov}\left( T_{k}\left( x_{j}\right),1\leq k\leq
d,1\leq j\leq n\right) \right] ^{1/2}}  \notag \\
&\leq &\frac{\left( 2\pi \right) ^{nd-1}2^{\frac{n\gamma +1}{2}}\Gamma
\left( \frac{n\gamma +1}{2}\right) }{\left[ \det \mathrm{Cov}%
\left(T_{0}\left( x_{1}\right) ,...,T_{0}\left( x_{n}\right) \right) \right]
^{d/2}}\prod\limits_{j=1}^{n}\left( \widetilde{\sigma }_{j}^{2}\right)
^{-\gamma/2},  \label{ineq:Jk1}
\end{eqnarray}
where $\widetilde{\sigma }_{j}^{2}$ is the conditional variance of $%
T_{k_{j}}\left( x_{j}\right) $ given $T_{l}\left( x_{i}\right) $ ($l\neq
k_{j}$ or $l=k_{j}$ but $i\neq j$) and $\Gamma \left( \cdot \right) $ is the
Gamma function. Now we define a permutation $\pi $ of $\left\{1,...,n\right%
\} $ such that $\pi \left( 1\right) =1,$ and 
\begin{equation*}
d_{\mathbb{S}^{2}}\left( x_{\pi ( j) },x_{\pi ( j-1)}\right) =\min \left\{
d_{\mathbb{S}^{2}}\left( x_{i},x_{\pi \left(j-1\right) }\right) ,i\in
\left\{ 1,...,n\right\} \backslash \left\{ \pi( 1) ,...,\pi ( j-1) \right\}
\right\} ,
\end{equation*}
Then by $\left( \ref{ineq:SLND}\right) $\ in Lemma \ref{C1}, we see that $%
K_{2,2}^{\gamma n/2}\prod\limits_{j=1}^{n}\frac{1}{\widetilde{\sigma }%
_{j}^{\gamma }}$ is bounded by 
\begin{eqnarray*}
\lefteqn{ \prod\limits_{j=1}^{n}\frac{1}{\min \left\{ \left[ \rho _{\alpha
}\left( d_{\mathbb{S}^{2}}\left( x_{\pi \left( j\right) },x_{i}\right)
\right) \right] ^{\gamma }: i\neq \pi \left( j\right) \right\} }} \\
&\leq &\prod\limits_{j=1}^{n} \frac{1}{\left[ \min \left\{ \frac{1}{2}%
\rho_{\alpha }\left( d_{\mathbb{S}^{2}}( x_{\pi ( j) },x_{\pi( j-1) })
\right) ,\rho _{\alpha }\left( d_{\mathbb{S}^{2}}( x_{\pi ( j) },x_{\pi (
j+1) })\right) \right\} \right] ^{\gamma }} \\
&\leq &\prod\limits_{j=1}^{n}\frac{4}{\left[ \rho _{\alpha }\left( d_{%
\mathbb{S}^{2}}( x_{\pi ( j) },x_{\pi ( j-1)}) \right) \right] ^{2\gamma }}%
\leq \prod\limits_{j=1}^{n}\frac{4}{\left[ \min_{1\leq i\leq j-1}\rho
_{\alpha }\left( d_{\mathbb{S}^{2}}(x_{\pi ( j) },x_{\pi ( i) }) \right) %
\right]^{2\gamma }}
\end{eqnarray*}
For the sake of notation's simplicity, we denote by $\pi ( j) =j$ for each $%
j=1,...,n.$ By $\left( \ref{ineq:Jk1}\right) $\ and Lemma \ref{Charac-Cov},
we derive that 
\begin{equation}  \label{ineq:Jk2}
J_{\mathbf{k}}\left( \mathbf{x}\right) \leq \frac{C_{3,2}^{n}\left(n-1%
\right) !}{\Pi _{j=2}^{n}\left[ \min_{1\leq i\leq j-1}\rho _{\alpha} \left(
d_{\mathbb{S}^{2}}( x_{j},x_{i}) \right) \right]^{d+2\gamma }}
\end{equation}
where $C_{3,2}$ is a positive constant depending on $K_{2,2}$ and $d.$
Therefore, by $( \ref{IntJ}) $ and $( \ref{ineq:Jk2}) $ above, we have 
\begin{equation}  \label{ineq:Ik1}
\mathbb{E}\left\{ \left[ L( \mathbf{t},D) -L( \mathbf{s},D) \right]
^{n}\right\} \leq \left( \frac{dC_{3,2}}{2^{\gamma }\pi }\right) ^{n}
\left\Vert \mathbf{t}-\mathbf{s}\right\Vert ^{n\gamma }n!\sum_{\mathbf{k}\in
\{ 1,...,d\} ^{n}}I_{\mathbf{k}},
\end{equation}
where 
\begin{equation}  \label{def:Ik}
I_{\mathbf{k}}=:\int_{D^{n}}\Pi _{j=2}^{n}\left[ \min_{1\leq i\leq
j-1}\rho_{\alpha }\left( d_{\mathbb{S}^{2}}\left( x_{j},x_{i}\right) \right) %
\right]^{-\left( d+2\gamma \right) }d\boldsymbol{\nu }.
\end{equation}
Similar to the argument in the proof of Lemma \ref{L1}, for any fixed $j\in
\left\{ 2,...,n\right\} ,$ we define the domains $\Gamma _{i},\ i=1,...,j-1,$
same as in $\left( \ref{GammaD}\right) ,$ then we can obtain 
\begin{eqnarray}
A_{j} &:=&\int_{D}\frac{1}{\left[ \min_{1\leq i\leq j-1}\rho _{\alpha}\left(
d_{\mathbb{S}^{2}}\left( x_{j},x_{i}\right) \right) \right]^{d+2\gamma }}%
d\nu \left( x_{j}\right)  \notag \\
&=&\sum_{i=1}^{j-1}\int_{\Gamma _{i}}\frac{1}{\left[ \min_{1\leq i\leq
j-1}\rho _{\alpha }\left( d_{\mathbb{S}^{2}}( x_{j},x_{i})\right) \right]
^{d+2\gamma }}d\nu \left( x_{j}\right)  \notag \\
&\leq &\sum_{i=1}^{j-1}\int_{0}^{2\pi }\int_{0}^{r_{i}\left( \phi \right) }%
\frac{\theta }{\left[ \rho _{\alpha }\left( \theta \right) \right]%
^{d+2\gamma }}d\theta d\phi =\sum_{i=1}^{j-1}\eta ^{-1}\int_{0}^{2\pi }\left[%
r_{i}\left( \phi \right) \right] ^{\eta }d\phi  \notag \\
&\leq &C_{3,3}\sum_{i=1}^{j-1}\left[ \int_{0}^{2\pi }\frac{1}{2}\left[%
r_{i}\left( \phi \right) \right] ^{2}d\phi \right] ^{\eta/2}
=C_{3,3}\sum_{i=1}^{j-1}\left[ \nu \left( \Gamma _{i}\right) \right]^{\eta
/2}  \label{intmin} \\
&\leq &C_{3,3}\left( j-1\right) \left( \frac{1}{\left( j-1\right) }
\sum_{i=1}^{j-1}\nu \left( \Gamma _{i}\right) \right) ^{\eta /2} \leq
C_{3,3}\left( j-1\right) \left( \frac{\nu \left( D\right) }{\left(j-1\right) 
}\right) ^{\eta /2}.  \label{ineq:Aj}
\end{eqnarray}%
where constant $\eta $ is the one defined in $\left( \ref{eta}\right) $ and $%
C_{3,3}=\eta ^{-1}\left( 2\pi \right) ^{1-\eta /2}$. The inequality in $%
\left( \ref{intmin}\right) $\ holds if and only if $0<\eta <2$, which is
always true for all $\alpha \in \left( 2,4\right) ,$ $\gamma \in
\left(0,1\right) ,\ d\in \mathbb{N}^{+}$ satisfying $\left( \alpha -2\right)
\left( d+2\gamma \right) <4.$ Hence, by the definition of $I_{\mathbf{k}}$
in $\left( \ref{def:Ik}\right) $ and the inequality $( \ref{ineq:Aj}) $
above, we have 
\begin{eqnarray}
I_{\mathbf{k}} &=&\int_{D}\left[ \int_{D^{n-1}}\Pi _{j=2}^{n}A_{_{j}}d\nu
\left( x_{n}\right) \cdots d\nu \left( x_{2}\right) \right] d\nu
\left(x_{1}\right)  \label{mineq2} \\
&\leq &( 2\pi ) ^{n\eta /2}C_{3,3}^{n}\left[ ( n-1) !\right] ^{1-\eta /2}\nu
( D) ^{( n-1) \eta /2+1}.  \notag
\end{eqnarray}
Thus, by $\left( \ref{ineq:Ik1}\right) $ and $\left( \ref{mineq2}\right) $
above, we immediately obtain 
\begin{eqnarray*}
\mathbb{E}\left\{ \left[ L\left( \mathbf{t},D\right) -L\left( \mathbf{s}%
,D\right) \right] ^{n}\right\} &\leq &\left( n-1\right) !\left\Vert \mathbf{t%
}-\mathbf{s}\right\Vert ^{n\gamma }\sum_{\mathbf{k}\in
\left\{1,...,p\right\} ^{n}}C_{3,4}^{n}r^{( n-1) \eta +2} \\
&\leq &K_{3,2}^{n}\left[ ( n-1) !\right] ^{2-\eta /2}\left\Vert \mathbf{t}-%
\mathbf{s}\right\Vert ^{n\gamma }\nu ( D) ^{(n-1) \eta /2+1}
\end{eqnarray*}
where the constants $C_{3,4}=C_{3,2}\left( 2\pi \right) ^{\eta /2}C_{3,3}$
and $K_{3,2}=dC_{3,4}.$ The proof is then completed.
\end{proof}

\begin{proof}[Proof of Theorem \textbf{\protect\ref{Th1}: }Joint Continuity]

It follows immediately from Kolmogorov's continuity theorem and Lemmas \ref%
{L1} and \ref{L2} that, for any two sets $D_{1},D_{2}\subset \mathbb{S}^{2}$
with $\nu \left( D_{i}\right) >0,i=1,2,$ and any $\mathbf{s},\mathbf{t}\in 
\mathbb{R}^{d}$, we have 
\begin{equation}  \label{ineq:JC1}
\left\vert L( \mathbf{t},D_{1}) -L( \mathbf{s},D_{1})\right\vert \leq
C_{3,5}\left\Vert \mathbf{t}-\mathbf{s}\right\Vert ^{\gamma}\nu ( D_{1})
^{\eta /2},
\end{equation}
and 
\begin{equation}  \label{ineq:JC2}
\left\vert L( \mathbf{s},D_{1}) -L( \mathbf{s},D_{2}) \right\vert \leq
C_{3,6}\nu ( \Delta D) ^{\left[ 1-( \alpha -2) d/4\right] },
\end{equation}
where $C_{3,5},C_{3,6}$ are positive constants depending only on $%
\alpha,A,d, $ and $C_{3,5}$depends on $\gamma $ as well. Moreover, the set $%
\Delta D=\left\{ D_{1}\cup D_{2}\right\} \backslash \left\{ D_{1}\cap
D_{2}\right\} $, that is, the union of $D_{1}$ and $D_{2},$ excluding their
intersection. Therefore, by $\left( \ref{ineq:JC1}\right) ,\left( \ref%
{ineq:JC2}\right) $ together with the following inequality 
\begin{equation*}
\left\vert L( \mathbf{t},D_{1}) -L( \mathbf{s},D_{2})\right\vert \leq
\left\vert L\left( \mathbf{t},D_{1}\right) -L\left( \mathbf{s},D_{1}\right)
\right\vert +\left\vert L( \mathbf{s},D_{2})-L( \mathbf{s},D_{2})
\right\vert ,
\end{equation*}
we have

\begin{enumerate}
\item[(i)] If replacing $D$ with $D\left( x,r\right) ,$ then for any $%
r_{1},r_{2}\in \left( 0,\delta \right) ,$ $x\in \mathbb{S}^{2},$ 
\begin{eqnarray*}
\lefteqn{ \left\vert L( \mathbf{t},D( x,r_{1}) ) -L(\mathbf{s},D( x,r_{2}) )
\right\vert } \\
&\leq &\pi C_{3,5}\left\Vert \mathbf{t}-\mathbf{s}\right\Vert
^{\gamma}r_{1}^{\eta }+\pi C_{3,6}\left\vert r_{1}^{2}-r_{2}^{2}\right\vert
^{\left[1-\left( \alpha -2\right) d/4\right] }
\end{eqnarray*}
and for any $r\in \left( 0,\delta \right) ,$ $x,y\in \mathbb{S}^{2}\ $with $%
d_{\mathbb{S}^{2}}\left( x,y\right) <2r,$ 
\begin{eqnarray*}
\lefteqn {\left\vert L\left( \mathbf{t},D\left( x,r\right) \right) -L\left( 
\mathbf{s},D\left( y,r\right) \right) \right\vert } \\
&\leq &\pi C_{3,5}\left\Vert \mathbf{t}-\mathbf{s}\right\Vert
^{\gamma}r^{\eta }+6C_{3,6}\left[ d_{\mathbb{S}^{2}}( x,y) r\right] ^{\left[
1-( \alpha -2) d/4\right] },
\end{eqnarray*}
where we have used the fact that 
\begin{eqnarray*}
\nu \left( D_{1}\cap D_{2}\right) &=&4\left( \frac{\pi }{2}-\arcsin \frac{d_{%
\mathbb{S}^{2}}\left( x,y\right) }{r}\right) (1-\cos r) \\
&&-d_{\mathbb{S}^{2}}\left( x,y\right) \sqrt{r^{2}-d_{\mathbb{S}^{2}}(x,y)
^{2}/4} +o\left( d_{\mathbb{S}^{2}}\left( x,y\right) r\right)
\end{eqnarray*}
and 
\begin{eqnarray*}
\nu ( \Delta D) &=& 2\pi (1-\cos r)-\nu \left( D_{1}\cap D_{2}\right) \\
&\leq & 6d_{\mathbb{S}^{2}}( x,y) r
\end{eqnarray*}
Here we have denoted by $o( \cdot ) $ the higher order terms.

\item[(ii)] If replacing $D$ with the angular section $V(\vartheta
,\varphi)=\{(\theta ,\phi ):0\leq \theta \leq \vartheta ,0\leq \phi \leq
\varphi \}$ with $(\vartheta ,\varphi )\in \lbrack 0,\pi ]\times \lbrack
0,2\pi )$, then for any 
\begin{eqnarray*}
\lefteqn{\left\vert L\left( \mathbf{t},V(\vartheta _{1},\varphi _{1})\right)
-L\left( \mathbf{s},V(\vartheta _{2},\varphi _{2})\right) \right\vert} \\
& \leq & \pi C_{3,5}\left\Vert \mathbf{t}-\mathbf{s}\right\Vert ^{\gamma
}\varphi_{1}^{\eta /2}\vartheta _{1}^{\eta } \\
&&+\pi C_{3,6}\left[ \left\vert \varphi _{1}-\varphi _{2}\right\vert \min
\left\{ \vartheta _{1}^{2},\vartheta _{2}^{2}\right\}+\min \left\{
\varphi_{1},\varphi _{2}\right\} \left\vert \vartheta
_{1}^{2}-\vartheta_{2}^{2}\right\vert \right] ^{\beta /4}
\end{eqnarray*}
The joint continuity of local times $L\left( \mathbf{t},D\right) $ \emph{%
w.r.t.} $\mathbf{t,}D$ in different cases is then obtained and hence the
proof of Theorem \ref{Th1} is completed.
\end{enumerate}
\end{proof}

\section{H\"{o}lder conditions of local times}

\label{Sec:Holder condition}

In order to prove Theorem \ref{Th2}, we need the following lemma which is
readily seen in view of Lemmas \ref{L1} and \ref{L2} :

\begin{lemma}
\label{Lemma:EL+r.v.} Under conditions of Theorem \ref{Th1}, there exist
positive constants $K_{4,1}$ $K_{4,2}$ depending on $\alpha ,\ K_{0},\ d$
and $K_{4,2}$ depends on $\gamma $ as well, such that for any open disk $%
D=D\left( x,r\right) \subset \mathbb{S}^{2}$ with $r\in \left(
0,\delta\right) $, any $\mathbf{t}\in \mathbb{R}^{d}$, $x_{0}\in \mathbb{S}%
^{2}$ and all even integers $n\geq 2,$ and $0<\gamma <1,$ 
\begin{equation*}
\mathbb{E}\left\{ \left[ L( \mathbf{t}+\mathbf{T}( x_{0}),D) \right]
^{n}\right\} \leq K_{4,1}^{n}( n!) ^{\frac{d}{4}( \alpha -2) }r^{\frac{n}{2}%
\left[ 4-( \alpha -2) d\right] },
\end{equation*}
\begin{equation*}
\mathbb{E}\left\{ \left[ L( \mathbf{t}+\mathbf{T}( x_{0},D) -L( \mathbf{s}+%
\mathbf{T}( x_{0}) ,D) \right]^{n}\right\} \leq K_{4,2}^{n}( n!) ^{2-\eta
/2}\left\Vert \mathbf{t}-\mathbf{s}\right\Vert ^{n\gamma }r^{n\eta },
\end{equation*}
where $\eta $ is the constant defined in $( \ref{eta}) $.
\end{lemma}

\begin{proof}
For any points $x_{1},...,x_{n}\in \mathbb{S}^{2},$ let $\mathbf{Z}(x_{j})=%
\mathbf{T}( x_{j}) -\mathbf{T}( x_{0}) ,$ $j=1,...,n$, we have 
\begin{eqnarray*}
\lefteqn{ \det \mathrm{Cov}\left[ Z_{0}\left( x_{1}\right) ,...,Z_{0}(x_{n}) %
\right] } \\
&=&\mathrm{Var}( Z_{0}( x_{1}) ) \prod\limits_{j=2}^{n}\mathrm{Var}( Z_{0}(
x_{j}) |Z_{0}( x_{1}),...,Z_{0}( x_{j-1}) ) \\
&\geq & K_{2,2}( K_{2,2}^{\prime }) ^{( n-1) /2}\rho_{\alpha }( x_{1},x_{0})
\prod\limits_{j=2}^{n}\min_{0\leq i\leq j-1}\rho _{\alpha }( x_{j},x_{i}) .
\end{eqnarray*}
in view of Corollary \ref{C2'}. Let $L_{\mathbf{Z}}( \mathbf{t},D) $ be the
local time of $\mathbf{Z}$ at $\mathbf{t}$ in $D,$ and recall (\ref{ELeq})
and (\ref{ineq:mineq3}), then we can obtain that 
\begin{eqnarray*}
\lefteqn{\mathbb{E}\left\{ \left[ L_{\mathbf{Z}}( \mathbf{t},D) \right]%
^{n}\right\} } \\
&\leq& \left( K_{2,2}^{^{\prime }}/K_{2,2}\right) ^{-1/2}(2\pi ) ^{-nd/2}(
K_{2,2}^{\prime }) ^{-nd/2} \\
&&\times \int_{D}\left[ \int_{D^{n-1}}\frac{d\nu ( x_{1}) \cdots d\nu (
x_{n-1}) }{\Pi _{j=2}^{n}\min_{0\leq i\leq j-1} \left[ \rho_{\alpha }( d_{%
\mathbb{S}^{2}}( x_{i},x_{j}) ) \right]^{d}}\right] \frac{d\nu ( x_{1}) }{%
\left[ \rho _{\alpha }(x_{1},x_{0}) \right] ^{d}} \\
&\leq & ( C_{4,1}) ^{n}n!\left( \frac{r^{2}}{n!}\right) ^{n\beta /4},
\end{eqnarray*}%
where $C_{4,1}$ is a positive constant depending on $K_{3,1},K_{2,2}$ and $%
K_{2,2}^{^{\prime }}.$ Likewise, recall (\ref{ineq:Ik1}) and (\ref{mineq2}),
we have 
\begin{eqnarray*}
\lefteqn{\mathbb{E}\left[ L_{\mathbf{Z}}( \mathbf{t},D) -L_{\mathbf{Z}}( 
\mathbf{s},D) \right] ^{n}} \\
&\leq& C_{4,2}\left\Vert \mathbf{t}-\mathbf{s}\right\Vert ^{n\gamma } \\
&& \times \int_{D} \left[ \int_{D^{n-1}}\frac{d\nu \left( x_{n}\right)
\cdots d\nu \left( x_{2}\right) }{\Pi _{j=2}^{n}\min_{0\leq i\leq j-1} \left[
\rho_{\alpha }\left( d_{\mathbb{S}^{2}}\left( x_{i},x_{j}\right) \right) %
\right]^{d+2\gamma }}\right] \frac{d\nu \left( x_{1}\right) }{\left[ \rho
_{\alpha}\left( x_{1},x_{0}\right) \right] ^{d+2\gamma }} \\
&\leq & ( C_{4,3}) ^{n}( n!) ^{2-\eta /2}\left\Vert \mathbf{t}-\mathbf{s}%
\right\Vert ^{n\gamma }r^{n\eta },
\end{eqnarray*}
where $C_{4,2}$ and $C_{4,3}$ are positive constants depending on $%
K_{2,2},K_{2,2}^{\prime },K_{3,2}$ and $\gamma ,d.$ The results in Lemma \ref%
{Lemma:EL+r.v.}\ is then derived by the fact that 
\begin{equation*}
L_{\mathbf{Z}}( \mathbf{t},D) =L( \mathbf{t+T}(x_{0}) ,D) .
\end{equation*}
\end{proof}

Based on Lemma \ref{Lemma:EL+r.v.}, we now follow the similar line as in the
proof of Lemma 2.7 in \cite{Xiao97}, and obtain the results below:

\begin{lemma}
\label{Lemma:PrLT+r.v.} Assume conditions of Theorem \ref{Th1} hold, there
exists a positive constant $K_{4,3},$ $K_{4,4}$ depending on $d,$ $\alpha
,K_{0}$ and $K_{4,4}$ depends on $\gamma $ as well, such that for any open
disk $D=D( x,r) \subset \mathbb{S}^{2}$ with $r\in \left(0,\delta \right) $,
any $u>0,$ $\mathbf{t}\in \mathbb{\mathbb{R}}^{d}$, $x_{0}\in \mathbb{S}^{2}$%
, 
\begin{equation}  \label{Pr:LT+r.v.}
\mathbb{P}\left\{ L( \mathbf{t}+\mathbf{T}( x_{0}) ,D)\geq \frac{K_{4,3}r^{2}%
}{\left[ \rho _{\alpha }( ur) \right] ^{d}}\right\} \leq e^{-1/u^{2}}.
\end{equation}
and 
\begin{equation}  \label{Pr:LTincrem+r.v.}
\mathbb{P}\left\{ \left\vert L( \mathbf{t}+\mathbf{T}(x_{0}) ,D) -L( \mathbf{%
s}+\mathbf{T}( x_{0}),D) \right\vert \geq \frac{K_{4,4}r^{2}\left\Vert 
\mathbf{t}-\mathbf{s}\right\Vert ^{\gamma }}{\left[ \rho _{\alpha }( ur) %
\right]^{d}u^{4\gamma }}\right\} \leq e^{-u^{-2}}.
\end{equation}
\end{lemma}

\begin{proof}
Let 
\begin{equation*}
\Lambda =\frac{L( \mathbf{t}+\mathbf{T}( x_{0}) ,D) }{r^{2}},\ u_{n}=\frac{1%
}{\sqrt{n}}
\end{equation*}
with $n\in\mathbb{N}^{+}.$ Then by Chebyshev's inequality and Lemma \ref%
{Lemma:EL+r.v.}, we have 
\begin{eqnarray*}
\mathbb{P}\left\{ \Lambda \geq \frac{K_{4,3}}{\left[ \rho _{\alpha
}\left(u_{n}r\right) \right] ^{d}}\right\} &\leq &\frac{\mathbb{E}\left[
\Lambda^{n}\right] \rho _{\alpha }^{n}( u_{n}r) }{( K_{3,4})^{n}}( n!) ^{%
\frac{d}{4}( \alpha -2) }r^{\frac{n}{2}\beta } \\
&\leq &\left( \frac{K_{4,1}}{K_{4,3}}\right) ^{n}\left( n!\right) ^{\frac{d}{%
4}\left( \alpha -2\right) }u_{n}^{\frac{n}{2}d\left( \alpha -2\right) }. \\
&\leq &\exp \left\{ -\left( \frac{\left( \alpha -2\right) d}{8}-\ln \frac{%
K_{4,1}}{K_{4,3}}\right) n\right\} \\
&\leq &\exp \left\{ -2/u_{n}^{2}\right\} ,
\end{eqnarray*}
where the last inequality follows from Stirling's formula, and the constants
satisfy $K_{4,3}=K_{4,1}\exp \left\{ 2-\frac{\left( \alpha -2\right) d}{8}%
\right\} .$ Now for any $u>0$ small enough, there exists $n\in\mathbb{N}%
^{+}, $ such that $u_{n+1}\leq u<u_{n},$ and 
\begin{eqnarray*}
\mathbb{P}\left\{ \Lambda \geq \frac{K_{4,3}}{\left[ \rho _{\alpha
}\left(ur\right) \right] ^{d}}\right\} &\leq & \mathbb{P}\left\{ \Lambda
\geq \frac{K_{4,3}}{\left[ \rho _{\alpha }\left( u_{n}r\right) \right] ^{d}}%
\right\} \\
&\leq& \exp \left\{ -2/u_{n}^{2}\right\} \leq \exp \left\{ -u^{-2}\right\} .
\end{eqnarray*}
Hence inequality $\left( \ref{Pr:LT+r.v.}\right) $ is derived.

The proof of estimation $\left( \ref{Pr:LTincrem+r.v.}\right) $ is similar
to the argument above and we omit it here. Thus the proof of Lemma \ref%
{Lemma:PrLT+r.v.} is completed.
\end{proof}

Now we sketch a proof for Theorem \ref{Th2}.

\begin{proof}[\textbf{Proof of Theorem \protect\ref{Th2}}]
The proof for the upper bounds in the inequalies $\left( \ref{ineq:Local
Holder of LT}\right) $ is based on a chaining argument and quite similar to
the proof of Theorem 1.1 and 1.2 in \cite{Xiao97} section 3 by replacing the
notations $X(t),$ $B(\tau ,2^{-n})$ with ones of $\mathbf{T}\left( x\right) ,
$ $D\left( x_{0},2^{-n}\right) $ for some $x_{0}\in \mathbb{S}^{2},$ and
Lemma 2.7, 3.1 in \cite{Xiao97} with Lemmas \ref{LemTupper} and \ref%
{Lemma:PrLT+r.v.} in this paper.

For the lower bounds in the inequities $(\ref{ineq:Local Holder of LT}),$ we
first let $\mathbf{I}$ be the closure of the set $\mathbf{T}(D)\mathbf{=}%
\left\{ \mathbf{T}(x),x\in D\right\} $ for any disk $D=D(z,r)\subset \mathbb{%
S}^{2}.$ Recall the definition of local time $\left( \ref{def:Local time}%
\right) $ or formula $\left( \ref{eq:LT-occup density}\right) ,$ we have for
any $\omega \in \Omega ,$ 
\begin{eqnarray}
\nu (D) &=&\mu _{D}(I,\omega )=\int_{\mathbf{I}}L(\mathbf{t},D,\omega )d%
\mathbf{t}  \notag \\
&\leq &L^{\ast }(D,\omega )\cdot \frac{\pi ^{d/2}}{\Gamma (d/2+1)}\left\vert
\sup_{x,y\in D}\left\Vert \mathbf{T}(x)-\mathbf{T}(y)\right\Vert \right\vert
^{d}.  \label{ineq:LT-mod}
\end{eqnarray}%
That is, for any $z\in \mathbb{S}^{2},$ 
\begin{eqnarray*}
\liminf_{r\rightarrow 0}\frac{L^{\ast }(D(z,r))}{\phi _{1}\left( r\right) }
&\geq &C_{4,4}\cdot \liminf_{r\rightarrow 0}\left\{ \frac{\rho _{\alpha }(r/%
\sqrt{\log \log r^{-1}})}{\sup_{x,y\in D(z,r)}\left\Vert \mathbf{T}\left(
x\right) -\mathbf{T}(y)\right\Vert }\right\} ^{d} \\
&\geq &C_{4,4}\left\{ \lim_{r\rightarrow 0}\sup_{x,y\in D(z,r)}\frac{%
\left\Vert \mathbf{T}(x)-\mathbf{T}(y)\right\Vert }{\rho _{\alpha }(r/\sqrt{%
\log \log r^{-1}})}\right\} ^{-d} \\
&=&C_{4,4}K_{5,3}^{-d},
\end{eqnarray*}%
in view of Proposition \ref{Prop:Loc non-dif} in Section \ref{Sec:non-diff},
which derives the lower bounds in the equalities $\left( \ref{ineq:Local
Holder of LT}\right) $ in Theorem \ref{Th2}. The positive constant $C_{4,4}$
depends only on $d.$ The proof is then completed.
\end{proof}

\section{Modulus of non-differentiability \label{Sec:non-diff}}

In this section, we establish local and global Chung's law of the iterated
logarithm, or we say, the mudulus of non-differentiability of the random
field $\mathbf{T}$, which are the essential of proving the lower bound of
maximum local time in Theorem \ref{Th2}.

Before giving the main results, we first introduce the following
band-limited random field. For any two integers $1\leq L<U\leq \infty ,$ we
define a random field as follows: 
\begin{equation*}
T_{0}^{L,U}( x) =\sum_{\ell =L}^{U}\sum_{m=-\ell }^{\ell }a_{\ell m}Y_{\ell
m}( x) ,\ x\in \mathbb{S}^{2}.
\end{equation*}
Observe that $T_{0}^{L,U}\left( x\right) $ and $T_{0}^{L^{\prime
},U^{\prime}}\left( x\right) $ are independent for $L<U<L^{\prime
}<U^{\prime }$ in view of the orthogonality properties of the Fourier
components of the field $T_{0}(x)$ and the assumption of Gaussianity.

Meantime, let $T_{0}^{\Delta }$ be the random field defined by 
\begin{equation*}
T_{0}^{\Delta }=T_{0}-T_{0}^{L,U}.
\end{equation*}
For any $r$ small, take $L=\left[ r^{-1}( B( r) )^{-\kappa _{1}}\right] $
and $U=\left[ r^{-1}( B( r) )^{1-\kappa _{1}}\right] ,$ where $B( r) $ is a
function such as $( \log \log r^{-1}) ^{\kappa _{2}}$ or $\left( \log
r^{-1}\right)^{\kappa _{2}}$ and the constants $\kappa _{1},\kappa _{2}$ are
to be determined. Here $[ \cdot ] $ denotes integer part as usual. Then we
have the following approximation for $T_{0}^{\Delta }:$

\begin{lemma}
\label{LemTtail} Under the same condition as in Theorem \ref{Th1}, there
exist positive constants $K_{5,1}$ and $K_{5,2}$ depending only on $\alpha $
and $K_{0},$ such that for any $0<r<\delta ,$ $0<\kappa _{1}\leq \frac{%
\alpha }{2}-1$ and 
\begin{equation*}
u > K_{5,1}( B( r) ) ^{-\kappa _{1}( 2-\frac{\alpha }{2}) }\sqrt{\log B( r) }%
\ r^{\frac{\alpha }{2}-1},
\end{equation*}
we have 
\begin{equation*}
\mathbb{P}\left\{ \sup_{x,y\in D( z,r) }\left\vert T_{0}^{\Delta}( x)
-T_{0}^{\Delta }( y) \right\vert \geq u\right\} \leq \exp \left( -\frac{1}{%
K_{5,2}}\frac{\left( B\left( r\right) \right) ^{\kappa _{1}( 4-\alpha )
}u^{2}}{r^{\alpha -2}}\right) .
\end{equation*}
\end{lemma}

\begin{proof}
Like in many other arguments in this paper, we start by introducing a
suitable Gaussian metric $d_{T_{\Delta }}$ defined on $D( z,r)\subset 
\mathbb{S}^{2}$ by 
\begin{equation*}
d_{T_{0}^{\Delta }}( x,y) :=\left[ \mathbb{E}\left\vert T_{0}^{\Delta }( x)
-T_{0}^{\Delta }( y) \right\vert ^{2}\right] ^{1/2}.
\end{equation*}
Once again, due to the fact that $d_{\mathbf{T}^{\Delta }}( x,y)\leq \sqrt{%
K_{2,1}}\rho _{\alpha }( x,y) ,$ a simple metric entropy argument yields 
\begin{equation*}
N_{d_{T_{0}^{\Delta }}}( D( z,r) ,\varepsilon ) \leq C_{5,1}\frac{r^{2}}{%
\epsilon ^{4/( \alpha -2) }},
\end{equation*}
with $C_{5,1}$ depending on $K_{2,1}.$ More precisely, recall 
\begin{equation*}
d_{T_{0}^{\Delta }}^{2}\left( x,y\right) =\left(
\sum_{\ell=0}^{L-1}+\sum_{\ell =U+1}^{\infty }\right) \frac{2\ell +1}{4\pi }%
C_{\ell}\left\{ 1-P_{\ell }(\cos \theta )\right\} .
\end{equation*}
where $\theta =d_{\mathbb{S}^{2}}\left( x,y\right) .$ Let $0<\theta <r,$ by
Lemma \ref{C1'} in the Appendix, we obtain 
\begin{eqnarray*}
d_{\mathbf{T}^{\Delta }}^{2}( x,y) &\leq & K_{A}(L^{4-\alpha }\theta
^{2}+U^{2-\alpha }) \\
&\leq & K_{A}\left[ B( r) ^{-\kappa _{1}( 4-\alpha )}+B( r) ^{-( 1-\kappa
_{1}) ( \alpha -2) }\right] r^{\alpha -2} \\
&\leq & K_{A}B( r) ^{-\kappa _{1}( 4-\alpha )}r^{\alpha -2}:=\left\vert h(
r) \right\vert ^{2}r^{\alpha -2},
\end{eqnarray*}
where 
\begin{equation*}
h( r) =:\sqrt{K_{A}}B( r) ^{-\frac{\kappa _{1}}{2}( 4-\alpha ) }.
\end{equation*}
Hence, if we let 
\begin{equation*}
\overline{d}:=\sup \left\{ d_{T_{0}^{\Delta }}( x,y) :x,y\in D( z,r)
\right\} ,
\end{equation*}
which obviously have $\overline{d}\leq h( r) r^{\frac{\alpha }{2}-1}$, then 
\begin{equation*}
\begin{split}
& \int_{0}^{\overline{d}}\sqrt{\log N_{d_{T_{0}^{\Delta }}}( D(z,r)
,\epsilon ) }\ d\epsilon \leq \int_{0}^{h( r) r^{\frac{\alpha }{2}-1}}\sqrt{%
\log C_{4,1}\frac{r^{2}}{\epsilon ^{4/(\alpha -2) }}}\ d\epsilon \\
& \leq \frac{2}{\sqrt{\alpha -2}}\sqrt{K_{2,1}}r^{\frac{\alpha }{2}-1}\int_{%
\sqrt{\log \frac{\sqrt{K_{2,1}}}{h\left( r\right) }}}^{+\infty }ud\left(
-e^{-u^{2}}\right) \\
& \leq \frac{4h\left( r\right) }{\sqrt{\alpha -2}}\sqrt{\log \frac{\sqrt{%
K_{2,1}}}{h\left( r\right) }}\ r^{\frac{\alpha }{2}-1}.
\end{split}%
\end{equation*}
By exploiting Lemma \ref{DudleyUB}, we immediately derive that, for any 
\begin{equation*}
u > C_{5,2}( B( r) ) ^{-\kappa _{1}( 2-\alpha /2) }\sqrt{\log B( r) }\
r^{\alpha /2-1}
\end{equation*}
with $C_{5,2}>0$ depending on $K_{2,5},K_{A},\alpha ,\kappa _{1},$ it holds
that 
\begin{equation*}
\mathbb{P}\left\{ \sup_{x,y\in D( z,r) }\left\vert T_{0}^{\Delta}( x)
-T_{0}^{\Delta }( y) \right\vert \geq u\right\} \leq \exp \left( -\frac{u^{2}%
}{4K_{2,5}^{2}\left\vert h( r)\right\vert ^{2}r^{\alpha -2}}\right)
\end{equation*}%
Letting $K_{5,1}=2K_{2,5}C_{5,2}$ and $K_{5,2}=4K_{2,5}^{2}K_{A},$ the proof
is then completed.
\end{proof}

Now let us focus on the Chung's law of the iterated logarithm.

\begin{proposition}
\label{Prop:Loc non-dif} Under the conditions of Theorem \ref{Th1}, there
exists positive constants $K_{4,2}$ such that for any $z\in \mathbb{S}^{2}$
with 
\begin{equation}  \label{eq:loc non-diff}
\lim_{r\rightarrow 0}\sup_{x,y\in D( z,r) }\frac{\left\Vert \mathbf{T}( x) -%
\mathbf{T}( y) \right\Vert }{\rho_{\alpha }\left( r/\sqrt{\log \log r^{-1}}%
\right) } =K_{5,3}.
\end{equation}
\end{proposition}

\begin{proof}
Due to Lemma 7.1.1 in Marcus and Rosen \cite{MRbook} and the fact that $%
\left\Vert \mathbf{T}( x) -\mathbf{T}( y) \right\Vert =\sqrt{d}%
|T_{0}(x)-T_{0}(y)|,$ we only need to prove the upper and lower bounds of
the following form: there exist positive and finite constants $C_{5,3}$ and $%
C_{5,4}$ such that 
\begin{equation}  \label{ineq:Llocmod}
\lim_{r\rightarrow 0}\sup_{x,y\in D( z,r) } \frac{|T_{0}(x)-T_{0}(y)|}{\rho
_{\alpha }\left( r/\sqrt{\log \log r^{-1}}\right) } \geq C_{5,3},\ \ 
\hbox{
a.s.}
\end{equation}
and 
\begin{equation}  \label{ineq:Ulocmod}
\lim_{r\rightarrow 0}\sup_{x,y\in D( z,r) }\frac{|T_{0}(x)-T_{0}(y)|}{\rho
_{\alpha }\left( r/\sqrt{\log \log r^{-1}}\right) } \leq C_{5,4},\ \ 
\hbox{
a.s. }
\end{equation}%
which implies $( \ref{eq:loc non-diff}) $\ with $K_{5,3}\in \lbrack
C_{5,3},C_{5,4}].$

Recall $\left( \ref{ineq:LT-mod}\right) $ and the upper bound of $\left( \ref%
{ineq:Local Holder of LT}\right) $ in Theorem \ref{Th2}$,$ we have for any $%
z\in \mathbb{S}^{2},$ 
\begin{eqnarray*}
&&\liminf_{r\to 0}\sup_{x,y\in D( z,r) }\frac{\left\Vert \mathbf{T}( x) -%
\mathbf{T}( y)\right\Vert }{\rho _{\alpha }\left( r/\sqrt{\log \log r^{-1}}%
\right) } \\
&\geq &C_{4,4}\left[ \limsup_{r\to 0}\frac{L^{\ast}( D( z,r) ) }{\phi _{1}(
r) }\right]^{-1} \geq \frac{C_{4,4}}{K_{3}}.
\end{eqnarray*}%
where $\phi _{1}$ is defined in $\left( \ref{psi1}\right) .$ One can verify
(cf. Lemma 7.1.6 in \cite{MRbook}) that this implies $\left( \ref%
{ineq:Llocmod}\right) .$

Now let us focus on the Proof of $\left( \ref{ineq:Ulocmod}\right) $. Let $%
B( r) =\left( \log \log r^{-1}\right) ^{\kappa _{2}}$ for any $r\in \left(
0,\delta \right) $ and $\kappa _{2}>0$ to be determined. Now we choose $%
r_{k}=( 2\log k) ^{-k},k=1,2,\cdots ,$ and $L_{k}=\left[ B\left(
r_{k}\right) ^{-\kappa _{1}}r_{k}^{-1}\right] $ as well as $U_{k}=\left[
B\left( r_{k}\right) ^{1-\beta }r_{k}^{-1}\right] ,$ where $\beta \in (0,%
\frac{\alpha }{2}-1]$ and $\left[ \cdot \right] $ denotes the integer part
as before. Obviously $L_{k}<U_{k}<L_{k+1}<U_{k+1}$ for any $k\in \mathbb{N}%
^{+}.$ We would like to prove that for some constant $C_{5,5}>0,$ 
\begin{equation}  \label{eq:UTmain}
\dsum\limits_{k=1}^{\infty }\mathbb{P}\left\{ \sup_{x,y\in D(z,r_{k})
}\left\vert T_{0}^{L_{k},U_{k}}( x)-T_{0}^{L_{k},U_{k}}( y) \right\vert \leq
C_{5,5}\rho _{\alpha}\left( \frac{r_{k}}{\sqrt{\log \log r_{k}^{-1}}}\right)
\right\} =\infty .
\end{equation}%
\label{UTmain} Thus, due to the independence of $T_{0}^{L_{k},U_{k}}$ for
different $k^{\prime }s,$ we have 
\begin{equation}
\limsup_{k\to\infty}\sup_{x,y\in D(z,r_{k}) }\left\vert T_{0}^{L_{k},U_{k}}(
x)-T_{0}^{L_{k},U_{k}}( y) \right\vert \leq C_{5,5}\rho _{\alpha}\left( 
\frac{r_{k}}{\sqrt{\log \log r_{k}^{-1}}}\right) ,\ a.s.
\end{equation}
in view of the Borel-Cantelli Lemma. The equation $\left( \ref{eq:UTmain}%
\right) $ can be derived by the fact that 
\begin{equation*}
d_{T_{0}^{L_{k},U_{K}}}=\sqrt{\mathbb{E}\left\vert T_{0}^{L_{k},U_{k}}(x)
-T_{0}^{L_{k},U_{k}}( y) \right\vert ^{2}} \leq d_{T_{0}},
\end{equation*}%
and thus, similar to the argument as in the proof of Lemma \ref{Lem:Small
Ball P}, we have for some constant $C_{4,5}>0,$ 
\begin{eqnarray*}
&&\mathbb{P}\left\{ \sup_{x,y\in D( z,r_{k}) }\left\vert
T_{0}^{L_{k},U_{k}}( x) -T_{0}^{L_{k},U_{k}}( y) \right\vert \leq
C_{5,5}\rho _{\alpha }\left( \frac{r_{k}}{\sqrt{\log \log r_{k}^{-1}}}%
\right) \right\} \\
&\geq & \exp \left( -\frac{2K_{2,8}}{C_{5,5}}\log k\right) =k^{-1/2},
\end{eqnarray*}
where we have taken $C_{5,5}=4\log eK_{2,8}.$

On the other hand, recall $T_{0}^{\Delta _{k}}=T_{0}-T_{0}^{L_{k},U_{k}}.$ 
\textit{Let } $B\left( r_{k}\right) ^{\kappa _{1}\left( 4-\alpha \right)}=(
\log \log r_{k}^{-1}) ^{\alpha /2},$ then it is readily seen that for any $%
k\in \mathbb{N}^{+},$ 
\begin{equation*}
\rho _{\alpha }\left( r_{k}/\sqrt{\log \log r_{k}^{-1}}\right) > (B( r_{k})
) ^{-\kappa _{1}( 2-\alpha /2) }\sqrt{\log B( r_{k}) }r_{k}^{( \alpha -2)
/2},
\end{equation*}
and thus, by Lemma \ref{LemTtail}, we have for any constant $C_{5,6}>0,$ 
\begin{eqnarray*}
&&\mathbb{P}\left\{ \sup_{x,y\in D( z,r_{k}) }\left\vert T_{0}^{\Delta
_{k}}( x) -T_{0}^{\Delta _{k}}( y)\right\vert >C_{5,6}\rho _{\alpha }\left(
r_{k}/\sqrt{\log \log r_{k}^{-1}}\right) \right\} \\
&\leq & \exp \left( -\frac{C_{5,6}^{2}}{K_{5,2}}\frac{\left(
B\left(r_{k}\right) \right) ^{\kappa _{1}\left( 4-\alpha \right) }}{\left(
\log\log r_{k}^{-1}\right) ^{\alpha /2-1}}\right) \leq \exp \left( -\frac{%
C_{5,6}^{2}}{K_{5,2}}\log \log r_{k}^{-1}\right) \\
&\leq &\exp \left( -\frac{C_{5,6}^{2}}{K_{5,2}}\log k\right) =k^{-2},
\end{eqnarray*}
where we have chosen $C_{5,6}^{2}\log e=2K_{5,3}$. \textit{\ Therefore, we
have } 
\begin{equation*}
\dsum\limits_{k=1}^{\infty }\mathbb{P}\left\{ \sup_{x,y\in D(z,r_{k})
}\left\vert T_{0}^{\Delta _{k}}( x) -T_{0}^{\Delta_{k}}( y) \right\vert >
C_{5,6}\rho _{\alpha }\left( \frac{r_{k}}{\sqrt{\log \log r_{k}^{-1}}}%
\right) \right\} <\infty ,
\end{equation*}
and again by the Borel-Cantelli Lemma, we have 
\begin{equation}  \label{UTtail}
\limsup_{r\to0}\sup_{x,y\in D(z,r_{k}) }\left\vert T_{0}^{\Delta _{k}}( x)
-T_{0}^{\Delta_{k}}( y) \right\vert \leq C_{5,6}\rho _{\alpha }\left( \frac{%
r_{k}}{\sqrt{\log \log r_{k}^{-1}}}\right) ,\ a.s.
\end{equation}
The inequality $\left( \ref{ineq:Ulocmod}\right) $ is then derived in view
of $( \ref{UTmain}) $ and $( \ref{UTtail}) .$
\end{proof}

\section{Appendix}

\label{appendix}

\begin{lemma}
\label{C1'} There exists a positive constant $K_{A}$ depending on $K_{0}$
and $\alpha $, such that for any $\theta >0$ small, and positive integers $L<%
\frac{K_{A}}{\theta },$ we have 
\begin{equation*}
\sum_{\ell =1}^{L}\frac{2\ell +1}{4\pi }C_{\ell }\left\{ 1-P_{\ell
}(\cos\theta )\right\} \leq K_{A}L^{4-\alpha }\theta ^{2}\text{ ,}
\end{equation*}
and, for any $U>1,$ 
\begin{equation*}
\sum_{\ell =U}^{\infty }\frac{2\ell +1}{4\pi }C_{\ell }\left\{
1-P_{\ell}(\cos \theta )\right\} \leq K_{A}U^{2-\alpha }.
\end{equation*}
\end{lemma}

\begin{proof}
Recall first that, from Condition (A), there exists a positive constant $%
K_{0}$ such that for all $\ell =1,2,...$ 
\begin{equation*}
K_{0}^{-1}\ell ^{-\alpha +1}\leq \frac{2\ell +1}{4\pi }C_{\ell } \leq
K_{0}\ell ^{-\alpha +1}.
\end{equation*}
We recall also the following\ Hilb's asymptotics results (see \cite{szego}%
,page 195, Theorem 8.21.6): for $K_{A}>0,$ we have uniformly 
\begin{equation*}
P_{\ell }(\cos \theta ) =\left\{ \frac{\theta }{\sin \theta }%
\right\}^{1/2}J_{0}((\ell +\frac{1}{2})\theta )+\delta _{\ell }(\theta ),
\end{equation*}
where 
\begin{equation*}
\delta _{\ell }(\theta )<<\left\{ 
\begin{array}{cl}
\theta ^{2}O(1) & \text{ for }0<\theta <\frac{K_{A}}{\ell } \\ 
\theta ^{1/2}O(\ell ^{-3/2}) & \text{ for }\theta >\frac{K_{A}}{\ell }%
\end{array}
\right. ;
\end{equation*}
Moreover, 
\begin{equation*}
\lim_{u\rightarrow 0}\frac{\left\{ 1-J_{0}(K_{A}u)\right\} }{K_{A}^{2}u^{2}}=%
\frac{1}{2}\text{.}
\end{equation*}
Thus, by using the fact that 
\begin{equation*}
\frac{\theta }{\sin \theta }-1=\frac{\theta ^{2}}{6}+O(\theta ^{3})\text{
,as }\theta \rightarrow 0\text{,}
\end{equation*}
we obtain that 
\begin{eqnarray*}
\sum_{\ell =1}^{L}\frac{2\ell +1}{4\pi }C_{\ell }\left\{ 1-P_{\ell
}(\cos\theta )\right\} &\leq & K_{0}\sum_{\ell =1}^{L}\ell ^{1-\alpha
}\left( \frac{\ell ^{2}}{K_{A}^{2}}-\frac{1}{6}\right) \theta ^{2} \\
&\leq & \frac{K_{0}}{\left( 4-\alpha \right) K_{A}^{2}}L^{4-\alpha
}\theta^{2}.
\end{eqnarray*}
On the other hand, recall that $1-P_{\ell }(\cos \theta )\leq 2$ uniformly
for all $\theta ,$ whence we have, for any $U>1,$ 
\begin{equation*}
\sum_{\ell =U}^{\infty }\frac{2\ell +1}{4\pi }C_{\ell }\left\{
1-P_{\ell}(\cos \theta )\right\} \leq K_{1,0}\sum_{\ell =U}^{\infty }\ell
^{1-\alpha}\leq \frac{2K_{1,0}}{\alpha -2}U^{2-\alpha }.
\end{equation*}
Let $K_{A}=\max \left\{ \sqrt[3]{\frac{K_{0}}{\left( 4-\alpha \right) }},%
\frac{2K_{0}}{\alpha -2}\right\} ,$ the proof is then completed.
\end{proof}

\end{document}